\pdfoutput=1
\documentclass[a4paper]{amsart} 
\usepackage{scrhack}
\usepackage[T1]{fontenc} 
\usepackage[utf8]{inputenc}
\usepackage[english]{babel}
\usepackage[babel]{csquotes}
\usepackage{lmodern} 
\usepackage[stretch=10]{microtype}
\usepackage{ifthen}
\usepackage{amssymb}
\usepackage{mathrsfs}
\usepackage{amsthm}
\usepackage{mathtools}
\usepackage{graphicx}
\usepackage{tikz}
\usetikzlibrary{positioning, calc}
\usepackage{tikz-cd}
\usepackage{fullpage}
\usepackage{fixltx2e} 
\usepackage[pdftex,
            pdfauthor={Jan Hesse},
            pdftitle={An equivalence between Frobenius algebras and
              Calabi-Yau categories},
             breaklinks=true]{hyperref}

\newcommand{\Z}{\mathbb{Z}}

\newcommand{\K}{\mathbb{K}}

\newcommand{\sK}{\mathscr{K}}

\newcommand{\cC}{\mathcal{C}}
\newcommand{\cD}{\mathcal{D}}

\newcommand{\cM}{\mathcal{M}}

\newcommand{\id}{\mathrm{id}}

\newcommand{\tr}{\mathrm{tr}}

\newcommand{\ot}{\otimes}
\newcommand{\eps}{\varepsilon}
\newcommand{\To}{\Rightarrow}

\newcommand{\fd}{\mathrm{fd}}
\newcommand{\fg}{\mathrm{fg}} 
\newcommand{\B}[1]{B #1}

\newcommand{\core}[1]{\sK ({#1})}
\newcommand{\Mod}[2]  
{
  \ifthenelse{\equal{#1}{}}{  			
		\ifthenelse{\equal{#2}{}}		
			{\mathrm{Mod}}{ 			
				{\mathrm{Mod} \textrm{-}#2}		
			}
	}{									
		\ifthenelse{\equal{#2}{}}		
			{{#1\textrm{-}\mathrm{Mod}}}{		
				{{#1\textrm{-}\mathrm{Mod}\textrm{-}#2}}	
			}
	}
}

\DeclareMathOperator{\Hom}{Hom}

\DeclareMathOperator{\Ob}{Ob}

\DeclareMathOperator{\Aut}{Aut}

\DeclareMathOperator{\Rep}{Rep}
\DeclareMathOperator{\Frob}{Frob}
\DeclareMathOperator{\CY}{CY}
\DeclareMathOperator{\Vect}{Vect}

\DeclareMathOperator{\ev}{ev}

\DeclareMathOperator{\End}{End}
\DeclareMathOperator{\Alg}{Alg}

\DeclareMathOperator{\Nat}{Nat}

\DeclareMathOperator{\Bicat}{Bicat}

\theoremstyle{definition}
\newtheorem{newdef}{Definition}[section]
\newtheorem{ex}[newdef]{Example}
\newtheorem{remark}[newdef]{Remark}

\theoremstyle{plain} 
\newtheorem{theorem}[newdef]{Theorem}
\newtheorem{lemma}[newdef]{Lemma}
\newtheorem{prop}[newdef]{Proposition}
\newtheorem{cor}[newdef]{Corollary}

\numberwithin{equation}{section}
\title{An equivalence between semisimple symmetric Frobenius algebras and Calabi-Yau categories} 
  \author{Jan Hesse}
    \address{Fachbereich Mathematik, Universit\"at Hamburg, Bereich Algebra und Zahlentheorie, Bundesstraße 55, D – 20 146 Hamburg}
\begin{document}
\begin{abstract}
  We show that the bigroupoid of semisimple symmetric Frobenius
  algebras over an algebraically closed field and the bigroupoid of
  Calabi-Yau categories are equivalent. To this end, we construct a
  trace on the category of finitely-generated representations of a
  symmetric, semisimple Frobenius algebra, given by the composite of
  the Frobenius form with the Hattori-Stallings trace.
  \end{abstract}
  \begin{flushright}
    ZMP-HH/16-21\\
    Hamburger Beitr\"age zur Mathematik Nr. 618 \\[1cm]
  \end{flushright}
  \maketitle
\section{Introduction}
The starting point for this paper is the weak 2-functor sending an
algebra to its category of representations.  If $\Alg_2$ denotes the
Morita bicategory and $\Vect_2$ is the bicategory of linear
categories, the weak 2-functor $\Rep: \Alg_2^\fd \to \Vect_2^\fd$
sending a semisimple algebra to its category of finitely-generated
modules is an equivalence between the fully-dualizable objects of the
bicategories $\Alg_2$ and $\Vect_2$, cf. \cite[Appendix A]{bdsv15-2}.

In this paper, we endow the objects of $\Alg_2^\fd$ with the
additional structure of a symmetric Frobenius algebra. We show that
the category of finitely-generated representations of a semisimple
symmetric Frobenius algebra carries a canonical structure of a
Calabi-Yau category as considered in \cite{morre-segal}. The
Calabi-Yau structure on the representation category is given by the
composite of the Hattori-Stallings trace with the Frobenius form. This
allows us to construct a 2-functor
\begin{equation}
   \Rep^{\fg} : \Frob \to \CY
\end{equation}
between the bigroupoid of semisimple symmetric Frobenius algebras $\Frob$, and
the bigroupoid of Calabi-Yau categories $\CY$. Our main result in
theorem \ref{thm:cy-frob-equivalence} shows that this 2-functor is an
equivalence of bigroupoids.

This result is related to topological quantum field theories as
follows: the Baez-Dolan cobordism hypothesis, proved by Lurie
\cite{Lurie09} in an $(\infty,n)$-categorical setting, asserts that a
framed, fully-extended, topological quantum field theory
is classified by its value on the positively framed point. However,
one needs more data to classify \emph{oriented} theories, which is
given by the datum of an homotopy fixed point of a certain
$SO(n)$-action on the target category.

In 2-dimensions, it is claimed in \cite{fhlt} that the structure of an
$SO(2)$ fixed point on the fully-dualizable objects of $\Alg_2$ is
given by a semisimple symmetric Frobenius algebra. Furthermore,
Schommer-Pries showed in \cite{schommerpries-classification} that the
bigroupoid $\Frob$ classifies fully extended oriented field theories
with target $\Alg_2$. Thus, semisimple symmetric Frobenius algebras should
correspond to homotopy fixed points of an $SO(2)$-action on
fully-dualizable objects of $\Alg_2$.

In her thesis, Davidovich \cite{davi11} observed that the
$SO(2)$-action on $\Alg_2^\fd$ is trivializable. Thus, it suffices to
consider the trivial $SO(2)$-action. This approach is taken in
\cite{hsv16}, where the bigroupoids of homotopy fixed points of the trivial
$SO(2)$-action on $\Alg_2^\fd$ and $\Vect_2^\fd$ are computed in a
purely bicategorical setting, and respectively found to be equivalent
to $\Frob$ and $\CY$. Thus, combining the results above, one
should expect that the two bigroupoids $\Frob$ and $\CY$ are
equivalent. In this paper, we establish the equivalence directly,
without referring to the cobordism hypothesis and homotopy fixed
points.

The paper is organized as follows: in section
\ref{sec:frob-alg-morita-context}, we recall the definition of the
bicategory $\Frob$ of semisimple, symmetric Frobenius algebras,
compatible Morita contexts and intertwiners, and of the bicategory
$\CY$ of Calabi-Yau categories. 

In section \ref{sec:rep-construction}, we construct a weak 2-functor
$\Rep^{\fg}$ which sends a semisimple symmetric Frobenius algebra to
its category of finitely-generated modules. We endow this category of
representations with the Calabi-Yau structure given by the composite
of the Frobenius form with the Hattori-Stallings trace in definition
\ref{cy-structure-rep}, and show in theorem
\ref{thm:cy-frob-equivalence} that this functor is an equivalence of
bicategories. Section \ref{sec:proving-equivalence} is devoted to the
proof of theorem \ref{thm:cy-frob-equivalence}.


Throughout the paper we use the following conventions: we will work
over an algebraically closed field $\K$. All Frobenius algebras
appearing will be symmetric.
\section*{Acknowledgments}
I would like to thank Christoph Schweigert and Alessandro Valentino
for inspiring discussions, and Ingo Runkel for pointing out lemma
\ref{lem:cy-additive}. I would also like to thank the referee for
her/his very thorough report, which helped to sharpen the focus of
this paper. The author is supported by the RTG 1670
\enquote{Mathematics inspired by String Theory and Quantum Field
  Theory}.
\section{Frobenius algebras and Calabi-Yau categories}
\label{sec:frob-alg-morita-context}
For the convenience of the reader, we recall the definitions of
compatible Morita contexts between symmetric Frobenius algebras. This
material has already appeared in \cite{schommerpries-classification}
and \cite{hsv16}.
\subsection{The bicategory of symmetric Frobenius algebras}
\begin{newdef}
  A \emph{Frobenius algebra} $(A,\lambda)$ over a field $\K$ consists of
  an associative, unital $\K$-algebra $A$, together with a linear map
  $\lambda: A \to \K$, so that the pairing
  \begin{align}
    \begin{aligned}
      A \ot_\K A &\to \K \\
      a \ot b & \mapsto \lambda(ab)
    \end{aligned}
  \end{align}
  is non-degenerate. A Frobenius algebra is called symmetric if
  $\lambda(ab)=\lambda(ba)$ for all $a$ and $b$ in $A$.
\end{newdef}
\begin{newdef}
  \label{def:morita-context}
  Let $A$ and $B$ be two algebras. A \emph{Morita context} $\cM$ consists of
  a quadruple $\cM:=( {_{B}M_{A}}, {_{A}N_{B}},\eps,\eta)$, where
  $_BM_A$ is a $(B,A)$-bimodule, $_AN_B$ is an $(A,B)$-bimodule, and
  \begin{equation}
    \begin{aligned}
      \eps&: { _{A}N \otimes_{B}{ M_{A}}} \to {_{A}A_{A}} \\
      \eta&: {_BB_B} \to{ _{B}M \otimes_{A} {N_B}}
    \end{aligned}
  \end{equation}
  are isomorphisms of bimodules, so that the two diagrams
  \begin{equation}
    \label{eq:def-morita-context-1}
    \begin{tikzcd}[column sep=large]
      {_BM} \ot_A {N_B} \ot_B {M_A} \rar{\id_M \ot \eps}  & {_BM} \otimes_A {A_A} \dar \\
      {_BB} \ot_B {M_A} \uar{\eta \otimes \id_M} \rar & {_BM_A}
    \end{tikzcd}
  \end{equation}
  \begin{equation}
    \label{eq:def-morita-context-2}
    \begin{tikzcd}[column sep=large]
      {_AN} \ot_B {M} \ot_A {N_B} \dar{\eps \otimes \id_N} & {_AN} \otimes_B {B_B} \dar \lar{\id_N \ot \eta} \\
      {_AA} \ot_A {N_B} \rar & {_AN_B}
    \end{tikzcd}
  \end{equation}
  commute.
\end{newdef}
These two conditions are not independent from each other, as the next
lemma proves.
\begin{lemma}[{\cite[Lemma 3.3]{bass}}]
\label{lem:morita-context-only-1-diagram}
  In the situation of definition \ref{def:morita-context}, diagram
  \eqref{eq:def-morita-context-1} commutes if and only if diagram
  \eqref{eq:def-morita-context-2} commutes.
\end{lemma}

Note that Morita contexts are the adjoint 1-equivalences in the
bicategory $\Alg_2$ of algebras, bimodules and intertwiners. These
form a category, where the morphisms are given by the following:
\begin{newdef}
  \label{def:morphism-of-morita-context}
  Let $\cM:=( {_{B}M_{A}}, {_{A}N_{B}},\eps,\eta)$ and $\cM':=(
  {_{B}{M'}_{A}}, {_{A}{N'}_{B}},\eps',\eta')$ be two Morita contexts
  between two algebras $A$ and $B$. A \emph{morphism of Morita contexts}
  consists of a morphism of $(B,A)$-bimodules $f:M \to M'$ and a
  morphism of $(A,B)$-bimodules $g:N \to N'$, so that the two diagrams
  \begin{equation}
    \begin{tikzcd}[row sep=large, column sep=large]
      _BM \ot_A N_B \rar{f \ot g} & {_BM' \ot_A N'_B} \\
      B \uar{\eta} \urar[swap]{\eta'}& {}
    \end{tikzcd}
    \qquad
    \begin{tikzcd}[row sep=large, column sep=large]
      _AN \ot_B M_A \rar{g \ot f} \dar[swap]{\eps} & _AN' \ot_B M'_A \dlar{\eps'} \\
      A & {}
    \end{tikzcd}
  \end{equation}
  commute.
\end{newdef}
If the algebras in question have the additional structure of a
symmetric Frobenius form $\lambda: A \to \K$, we would like to
formulate a compatibility condition between the Morita context and the
Frobenius forms. The next lemma helps us to do that:

\begin{lemma}[{\cite[Lemma 3.71]{schommerpries-classification} or \cite[Lemma 2.3]{hsv16}}]
  \label{lem:A/[A,A]-B[B,B]-iso}
  Let $A$ and $B$ be two algebras, and let $( {_{B}M_{A}},
  {_{A}N_{B}},\eps,\eta)$ be a Morita context between $A$ and $B$.
  Then, there is a canonical isomorphism of vector spaces
  \begin{align}
    \begin{aligned}
      f: A / [A,A] & \to B /[B,B] \\
      [a] &\mapsto \sum_{i,j} \left[\eta^{-1}(m_j.a \ot n_i) \right]
    \end{aligned}
  \end{align}
  where $n_i$ and $m_j$ are defined by
  \begin{equation}
    \eps^{-1}(1_A)= \sum_{i,j} n_i \ot m_j \in N \ot_B M .
  \end{equation}
\end{lemma}

The isomorphism $f$ described in Lemma \ref{lem:A/[A,A]-B[B,B]-iso}
allows to introduce the following relevant definition.
\begin{newdef}
  \label{def:morita-compatible}
  Let $(A,\lambda^A)$ and $(B,\lambda^B)$ be two symmetric Frobenius
  algebras, and let $( {_{B}M_{A}}, {_{A}N_{B}},\eps,\eta)$ be a
  Morita context between $A$ and $B$. Since the Frobenius algebras are
  symmetric, the Frobenius forms necessarily factor through $A/[A,A]$
  and $B/[B,B]$.  We call the Morita context \emph{compatible} with
  the Frobenius forms, if the diagram
  \begin{equation}
    \label{eq:morita-compatible}
    \begin{tikzcd}
      A/[A,A] \ar{dr}[swap]{\lambda^A} \ar{rr}{f} && B/[B,B] \ar{dl}{\lambda^B} \\
      & \K &
    \end{tikzcd}
  \end{equation}
  commutes.  Using the notation from lemma
  \ref{lem:A/[A,A]-B[B,B]-iso}, this means that
  \begin{equation}
    \label{eq:morita-compatible-form}
    \lambda^A([a])=\sum_{i}\lambda^B \left( \left[\eta^{-1}(m_i.a \ot n_i) \right] \right)
  \end{equation}
  for all $a \in A$.
\end{newdef}

\begin{newdef}
  \label{def:frob-bicategory}
  Let $\Frob$ be the bicategory where
  \begin{itemize}
  \item objects are given by semisimple, symmetric Frobenius algebras,
  \item 1-morphisms are given by compatible Morita contexts, as in
    definition \ref{def:morita-compatible},
  \item 2-morphisms are given by isomorphisms of Morita contexts.
  \end{itemize}
  Note that $\Frob$ has got the structure of a symmetric monoidal
  bigroupoid, where the monoidal product is given by the tensor
  product over the ground field, which is the monoidal unit.
\end{newdef}

\subsection{The bicategory of Calabi-Yau categories }
Another main player of this paper are Calabi-Yau categories, which we
define next.

Let $\K$ be a field, and let $\Vect$ be the category of $\K$-vector
spaces. Recall the following terminology: a linear category is an
abelian category with a compatible enrichment over $\Vect$. A linear
functor is a right-exact additive functor which is also a functor of
$\Vect$-enriched categories.
\begin{newdef}
  \label{def:finite-cat}
  Following \cite[Appendix A]{bdsv15-2}, we call a linear category
  $\cC$ \emph{finite}, if
  \begin{enumerate}
  \item there are only finitely many isomorphism classes of simple
    objects of $\cC$,
  \item the category $\cC$ has enough projectives, 
  \item every object of $\cC$ has finite length, and
  \item the $\Hom$-spaces of $\cC$ are finite-dimensional.
  \end{enumerate}
\end{newdef}

\begin{newdef}
  \label{def:calabi-yau-cat}
  A \emph{Calabi-Yau category} $(\cC, \tr^\cC)$ is a linear, finite,
  semisimple category $\cC$, together with a family of linear
  maps
  \begin{equation}
    \tr_c^\cC: \End_\cC(c) \to  \K
  \end{equation}
  for each object $c$ of $\cC$, so that:
  \begin{enumerate}
  \item for each $f \in \Hom_\cC(c,d)$ and for each $g \in
    \Hom_\cC(d,c)$, we have that
    \begin{equation}
      \tr^\cC_c(g \circ f) = \tr^\cC_d (f \circ g), 
    \end{equation}
    \item for all objects $c$ and $d$ of $\cC$, the induced pairing
    \begin{align}
      \begin{aligned}
        \langle - \, , - \rangle_\cC: \Hom_\cC(c,d) \ot_\K \Hom_\cC(d,c) &\to \K      \\
        f \ot g& \mapsto \tr_c^\cC(g \circ f)
      \end{aligned}
    \end{align}
    is a non-degenerate pairing of $\K$-vector spaces.
  \end{enumerate}
  We will call the collection of morphisms $\tr_c^\cC$ a trace on
  $\cC$. Note that an equivalent way of defining a Calabi-Yau
  structure on a linear category $\cC$ is by specifying a natural
  isomorphism
  \begin{equation}
    \Hom_\cC(c,d) \to \Hom_\cC(d,c)^*,
  \end{equation}
  cf. \cite[Proposition 4.1]{schau12}.
\end{newdef}
\begin{remark}
  A Calabi Yau category with one object is a semisimple symmetric Frobenius
  algebra. More generally, the space of all endomorphisms of an object
  of a Calabi-Yau category has got the structure of a semisimple symmetric
  Frobenius algebra.
\end{remark}
We begin with two preparatory lemmas:
\begin{lemma}
\label{lem:cy-additive}
  Let $(\cC, \tr^\cC)$ be a Calabi-Yau category. Then, the trace is
  automatically additive: for each $f \in \End_\cC(x)$ and each $g \in \End_\cC(y)$, we
    have that
    \begin{equation}
      \tr_{x \oplus y }^\cC (f \oplus g) = \tr_x^\cC(f) + \tr_y^\cC(g).
    \end{equation}
\end{lemma}
\begin{proof}
  Denote by $p_x:x \oplus y \leftrightarrow x : \iota_x$ and by $p_y:x \oplus y
  \leftrightarrow y : \iota_y$ the canonical projections and
  inclusions. Then,
  \begin{equation}
    \id_{x \oplus y} = \iota_x \circ p_x + \iota_y \circ p_y. 
  \end{equation}
Then, by the linearity and cyclicity of the trace,
\begin{equation}
  \begin{aligned}
    \tr^\cC_{x \oplus y}(f+g) &= \tr^\cC_{x \oplus y} ((f+g) \circ
    (\iota_x \circ p_x + \iota_y \circ p_y)) \\
    &= \tr^\cC_{x \oplus y} ((f+g) \circ \iota_x \circ p_x )
    +\tr^\cC_{x \oplus y} ((f+g) \circ \iota_y \circ p_y) \\
&=\tr^\cC_x(p_x \circ (f+g) \circ \iota_x) + \tr^\cC_y(p_y \circ (f+g)
\circ \iota_y) \\
&=\tr^\cC_x(f) + \tr_y^\cC(g).
  \end{aligned}
\end{equation}
\end{proof}
Adapting the proof of \cite[Section 1]{Stallings65} to the setting of
linear categories establishes the following lemma:
\begin{lemma}
   \label{lem:trace-matrix-form}
  Let $(\cC, \tr^\cC)$ be a Calabi-Yau category, and let $x_1, \ldots,
  x_n$ be objects of $\cC$. Let $x:= \oplus_{i=1}^n x_i$, and let $f
  \in \End_\cC(x)$. Since $\cC$ is an additive category, we may write
  the morphism $f$ in matrix form as
  \begin{equation}
    f=
    \begin{pmatrix}
      f_{11} & \ldots & f_{1n} \\
      \vdots & & \vdots \\
      f_{n1} & \ldots & f_{nn}
    \end{pmatrix}
  \end{equation}
  where the entries $f_{ij}$ are morphisms $f_{ij} \in \Hom_\cC(x_j,
  x_i)$.  Then,
  \begin{equation}
    \tr_x^\cC(f) = \sum_{i=1}^n \tr_{x_i}^\cC(f_{ii}).
  \end{equation}
\end{lemma}

\begin{newdef}
  \label{def:calabi-yau-functor}
  Let $(\cC, \tr^\cC)$ and $(\cD, \tr^\cD)$ be two Calabi-Yau
  categories. A linear functor $F: \cC \to \cD$ is called a \emph{Calabi-Yau
  functor}, if
  \begin{equation}
    \tr_c^\cC(f)=\tr_{F(c)}^\cD(F(f))
  \end{equation}
  for each $c \in \Ob(\cC)$ and each $f\in \End_\cC(c)$.
  Equivalently, one may require that
  \begin{equation}
    \langle F f, Fg \rangle_\cD = \langle f, g \rangle_\cC 
  \end{equation}
  for every pair of morphisms $f:c \to c'$ and $g: c' \to c$ in $\cC$.

  If $F$, $G: \cC \to \cD$ are two Calabi-Yau functors between
  Calabi-Yau categories, a natural transformation of Calabi-Yau
  functors is just an ordinary natural transformation.
\end{newdef}

\begin{newdef}
  Let $\CY$ be the symmetric monoidal bigroupoid of Calabi-Yau
  categories, consisting of
  \begin{description}
  \item[objects] Calabi-Yau categories, 
  \item[1-morphisms] Equivalences of Calabi-Yau categories as in
    definition \ref{def:calabi-yau-functor},
  \item[2-morphisms] Natural isomorphisms.
  \end{description}
The monoidal structure of the bigroupoid $\CY$ is given by the Deligne tensor product of
finite, abelian categories.
\end{newdef}

\begin{lemma}
  \label{lem:cy-classification}
  Let $\cC$ be a finite semisimple linear category with $n$ simple
  objects over a field $\K$. Then, $\cC$ has got a structure of a
  Calabi-Yau category. Furthermore, the set of Calabi-Yau structures
  on $\cC$ stands in bijection to $(\K^*)^n$.
\end{lemma}
\begin{proof}
  If $\cC$ has got the structure of a Calabi-Yau category, the trace
  $\tr^\cC$ will be additive by lemma
  \ref{lem:trace-matrix-form}. Hence, the trace $\tr^\cC$ is uniquely
  determined by the endomorphism algebras of the simple objects. If
  $X$ is a simple object of $\cC$, Schur's lemma shows that
  $\End_\cC(X) \cong \K$ as vector spaces, since the ground field $\K$
  was assumed to be algebraically closed and $\cC$ is finite. One now
  checks that choosing
  \begin{equation}
    \tr^\cC_X : \End_\cC(X) \cong \K \to \K
  \end{equation}
  to be the identity for every simple object $X$ indeed defines the
  structure of a Calabi-Yau category on $\cC$. This shows the first
  claim.

  Now note that for a simple object $X$, due to its symmetry the trace
  $\tr^\cC_X$ is unique up to multiplication with an invertible
  central element in $Z(\End_\cC(X)) \cong \K$. Thus, the trace
  $\tr^\cC_X$ on $\End_\cC(X)$ is unique up to a non-zero element in
  $\K$. Taking direct sums now shows the second claim.
\end{proof}

\section{Constructing the equivalence}
\label{sec:rep-construction}
The purpose of this section is to construct a weak 2-functor $\Rep^{\fg}:
\Frob \to \CY$, which sends a semisimple, symmetric Frobenius algebra
to its category of finitely generated modules.
\subsection{A Calabi-Yau structure on the representation category of a
  Frobenius algebra}
First, we will show that the category of finitely generated modules
over a semisimple symmetric Frobenius algebra has the structure of a
Calabi-Yau category in the sense of definition
\ref{def:calabi-yau-cat}, and thus construct $\Rep^{\fg}$ on objects.
 
Let us first recall some standard material about finitely generated
and projective $R$-modules. If $M$ is a left $A$-module, the dual
module $M^*:=\Hom_A(M,A)$ is a right $A$-module with right action given
by $(f.a)(m):=f(m).a$.
\begin{lemma}[Dual basis lemma]
  \label{lem:dual-basis}
  Let $R$ be a commutative ring, let $A$ be a $R$-algebra, and let $P$
  be a left $A$-module. The following are equivalent:
  \begin{enumerate}
  \item The module $P$ is finitely generated and projective.
  \item There are $f_1, \ldots, f_n \in P^*$ and $p_1, \ldots, p_n \in
    P$ (sometimes called dual- or projective basis of $P$) so that
    \begin{equation}
      \label{eq:def-dual-basis}
      x = \sum_{i=1}^n f_i (x).p_i
    \end{equation}
    for all $x \in P$.
  \item The map
    \begin{align}
      \begin{aligned}
        \Psi_{P,P} : P^* \ot_A P &\to \End_A(P) \\
        f \ot p & \mapsto (x \mapsto f(x).p)
      \end{aligned}
    \end{align}
    is an isomorphism of $R$-modules.
  \item For any other left $A$-module $M$, the map
    \begin{align}
      \label{eq:def-psi}
      \begin{aligned}
        \Psi_{P,M} : P^* \ot_A M &\to \Hom_A(P,M) \\
        f \ot m & \mapsto (x \mapsto f(x).m)
      \end{aligned}
    \end{align}
    is an isomorphism of $R$-modules.
  \end{enumerate}
\end{lemma}
\begin{proof}
  The equivalence of $(1)$ and $(2)$ is proven in \cite[Lemma
  2.9]{lam2012lectures}. The implication $(1) \To (4)$ is
  \cite[Proposition 2.32]{adkins-weintraub-1992algebra}.

  $(4) \To (3)$ is trivial, since we may choose $M:=P$.

  $(3) \To (2)$: Suppose that $\Psi_{P,P}: P^* \ot_A P \to \End_A(P)$
  is an isomorphism. Then, a quick calculation confirms that
  $\Psi_{P,P}^{-1}(\id_P)$ is a dual basis.
\end{proof}

\begin{cor}
  \label{cor:psi-invertible-over-k}
  Let $A$ be a semisimple algebra, and let $M$ be a
  finitely generated $A$-module. Then, $\Psi_{M,M}: M^* \ot_A M
  \to \End_A(M) $ is an isomorphism.
\end{cor}
\begin{proof}
 As every $A$-module is projective by assumption, the corollary
 follows directly from lemma \ref{lem:dual-basis}.
\end{proof}
This corollary enables us to define a trace for finitely-generated
modules over a semisimple symmetric Frobenius algebra.


\begin{newdef}
\label{cy-structure-rep}
Let $(A,\lambda)$ be a semisimple symmetric Frobenius algebra with
Frobenius form $\lambda: A \to \K$. Let $M$ be a finitely-generated
left $A$-module. Denote by
  \begin{align}
    \begin{aligned}
      \ev: M^* \ot_A M &\to A \\
      f \ot m & \mapsto f(m)
    \end{aligned}
  \end{align}
  the evaluation.

  Since $M$ is finitely generated, the map $\Psi_{M,M}: \End_A(M) \to
  M^* \ot_A M$ is an isomorphism by corollary
  \ref{cor:psi-invertible-over-k}. We define a \emph{trace}
  $\tr_M^\lambda: \End_A(M) \to \K$ by the composition
  \begin{align}
    \label{eq:def-trace}
    \tr_M^\lambda: \End_A(M) \xrightarrow{\Psi_{M,M}^{-1}} M^* \ot_A M
    \xrightarrow{\ev} A \xrightarrow{\lambda} \K.
  \end{align}
\end{newdef}

\begin{remark}
  As defined here, the trace $\tr_M^\lambda$ is the composition of the
  Hattori-Stallings trace with the Frobenius form $\lambda$. For more
  on the Hattori-Stallings trace, see \cite{hattori1965},
  \cite{Stallings65} and \cite{bass-euler76}.
\end{remark}
\begin{ex}
  Let $(A, \lambda)$ be a semisimple symmetric Frobenius algebra.
  Suppose that $F$ is a free $A$-module with basis $e_1, \ldots,
  e_n$. Then,
    \begin{equation}
      \tr^\lambda_F(\id_F) = n \lambda(1_A).  
    \end{equation}
  \end{ex}
\begin{ex}
  \label{ex:rank-of-k^n}
  As a second example, and let $A:= M_n(\K)$ be the algebra of $n
  \times n$-matrices over $\K$ with Frobenius form $\lambda$ given by
  the usual trace of matrices. Then, $M:=\K^n$ is a projective (but
  not free), simple $A$-module. We claim:
  \begin{equation}
    \tr^\lambda_{M}( \id_M) = 1.
  \end{equation}
  Indeed, let $e_1, \ldots, e_n$ be a vector space basis of
  $\K^n$. This basis also generates $\K^n$ as an $A$-module. Define
  for each $1 \leq i \leq n$ a $\K$-linear map $f_i^* : \K^n \to
  M_n(\K)=A$ by setting
  \begin{equation}
    f_i^*(e_k) := \delta_{i,1} E_{k,1},
  \end{equation}
  where $E_{k,1}$ is the square matrix with $(k,1)$-entry given by one
  and zero otherwise.  A short calculation confirms that the $f_i^*$
  are even morphisms of $A$-modules. 
  Next, we claim that
  \begin{equation}
    \Psi^{-1}_{M,M} = \sum_{i=1}^n f_i^* \ot e_i \in M^* \ot_A M.
  \end{equation}
  Indeed,
  \begin{equation}
    \begin{aligned}
      \Psi_{M,M} \left( \sum_{i=1}^n f_i^* \ot e_i
      \right) (e_k) &= \sum_{i=1}^n f_i^*(e_k).e_i \\
      & = \sum_{i=1}^n
      \delta_{i,1} E_{k,1} e_i \\
      & = E_{k,1} e_1 =e_k.
    \end{aligned}
  \end{equation}
  Thus,
  \begin{equation}
    \tr^\lambda_M( \id_M) =  \lambda \left( \sum_{i=1}^n  f_i^*(e_i) \right) = \lambda \left(\sum_{i=1}^n E_{i,1} \delta_{1,i} \right) =\lambda (E_{1,1}) =1.
  \end{equation}
\end{ex}

Next, we show that $\tr_M^\lambda$ has indeed the properties of a
trace. In order to show that the trace is symmetric, we need an
additional lemma first, which can be proven by a small calculation.

\begin{lemma}
  \label{lem:psi-composition}
  Let $A$ be an $\K$-algebra, and let $M$ and $N$ be left
  $A$-modules. Define a linear map
  \begin{equation}
    \begin{aligned}
      \xi: (M^* \ot_A N) \times (N^* \ot_A M) &\to M^* \ot_A M \\
      (f \ot n, g \ot m) & \mapsto f \ot g(n).m
    \end{aligned}
  \end{equation}
  Then, the following diagram commutes:
  \begin{equation}
    \begin{tikzcd}
      (M^* \ot_A N) \times (N^* \ot_A M) \rar{\xi} \dar{\Psi_{M,N} \times \Psi_{N,M}} & M^* \ot_A M \dar{\Psi_{M,M}}\\
      \Hom_A(M,N) \times \Hom_A(N,M) \rar{\circ} & \Hom_A(M,M).
    \end{tikzcd}
  \end{equation}
  Here, the horizontal map at the bottom is given by composition of
  morphisms of $A$-modules and $\Psi_{M,M}$ is defined as in equation
  \eqref{eq:def-psi}.
\end{lemma}
We are now ready to show that the trace is symmetric:
\begin{lemma}
  \label{lem:trace-symmetric}
  Let $(A, \lambda)$ be a semisimple, symmetric Frobenius algebra.
  Let $M$ and $N$ be finitely-generated $A$-modules, and let $f:M \to
  N$ and $g: N \to M$ be morphisms of $A$-modules. Then,
  \begin{equation}
    \tr_M^\lambda(g \circ f) = \tr_N^\lambda(f \circ g).
  \end{equation}
\end{lemma}
\begin{proof}
  Suppose that
  \begin{equation}
    \begin{aligned}
      \Psi_{M,N}^{-1}(f)&= \sum_{i,j} m_i^* \ot n_j \in M^* \ot_A N \qquad \text{and}\\
      \Psi_{N,M}^{-1}(g)&= \sum_{k,l} x_k^* \ot y_l \in N^* \ot_A M.
    \end{aligned}
  \end{equation}
  We calculate:
  \begin{align}
    \label{eq:68}
    \begin{aligned}
      \tr_M^\lambda(g \circ f) &= (\lambda \circ \ev \circ \Psi_{M,M}^{-1})(g \circ f)   \\ & =  (\lambda \circ \ev) \left( \sum_{i,j,k,l} m_i^* \ot x_k^*(n_j).y_l \right) \qquad &&  \text{(by lemma  \ref{lem:psi-composition})} \\ &= \lambda \left(\sum_{i,j,k,l} m_i^*(x_k^*(n_j).y_l) \right) \\
      &=\sum_{i,j,k,l} \lambda(x_k^*(n_j) \cdot m_i^*(y_l)).
    \end{aligned}
    \intertext{On the other hand,}
    \begin{aligned}
      \label{eq:69}
      \tr_N^\lambda(f\circ g) &= (\lambda \circ \ev \circ \Psi_{N,N}^{-1})(f \circ g) \\
      & =(\lambda \circ \ev) \left( \sum_{i,j,k,l} x_k^* \ot m_i^*(y_l).n_j \right)  \qquad &&  \text{(by lemma  \ref{lem:psi-composition})}\\
      & = \lambda \left(\sum_{i,j,k,l}x_k^*(m_i^*(y_l).n_j) \right) \\
      & = \sum_{i,j,k,l} \lambda(m_i^*(y_l) \cdot x_k^*(n_j)).
    \end{aligned}
  \end{align}
  Since $\lambda$ is symmetric, the right hand-sides of equations
  \eqref{eq:68} and \eqref{eq:69} agree. This shows that the trace is
  symmetric.
\end{proof}

Next, we would like to show that the trace is non-degenerate. We first
recall a preparatory lemma.
\begin{lemma}[{\cite[Lemma 2.2.11]{kock2004frobenius}}]
  \label{lem:symmetric-frobenius-form-unique-up-to-multiple}
  Let $(A, \lambda)$ be a symmetric Frobenius algebra. Then, every
  other symmetric Frobenius form on $A$ is given by multiplication
  with a central invertible element of $A$.
\end{lemma}
We are now ready to show that the trace is non-degenerate.
\begin{lemma}
  \label{trace-non-deg}
  Let $(A, \lambda)$ be a semisimple, symmetric Frobenius algebra, and
  let $M$ and $N$ be finitely-generated $A$-modules. Then, the
  bilinear pairing of vector spaces induced by the trace in definition
  \ref{cy-structure-rep}
  \begin{equation}
    \begin{aligned}
      \langle - , - \rangle : \Hom_A(M,N) \times \Hom_A(N,M) & \to \K \\
      (f,g) & \mapsto \tr_M^\lambda(g \circ f)
    \end{aligned}
  \end{equation}
  is non-degenerate.
\end{lemma}

\begin{proof}
  By Artin-Wedderburn's theorem, the algebra $A$ is isomorphic to a
  direct product of matrix algebras over $\K$:
  \begin{equation}
    A \cong \prod_{i=1}^r M_{n_i}(\K).
  \end{equation}
  Since the sum of the usual trace of matrices gives each $A$ the
  structure of a symmetric Frobenius algebra, lemma
  \ref{lem:symmetric-frobenius-form-unique-up-to-multiple} shows that
  the Frobenius form $\lambda$ of $A$ is given by
  \begin{equation}
    \lambda = \sum_{i=1}^r \lambda_i \tr_i,
  \end{equation}
  where $\tr_i : M_{n_i}(\K) \to \K$ is the usual trace of matrices
  and $\lambda_i \in \K^*$ are non-zero scalars.

  Recall that a module over a finite-dimensional algebra is
  finite-dimensional (as a vector space) if and only if it is finitely
  generated as a module, cf. \cite[Proposition
  2.5]{frobenius-algebras-i}.  A classical theorem in representation
  theory (cf. theorem 3.3.1 in \cite{eghlsvy11}) asserts that the only
  finite-dimensional simple modules of $A$ are given by $V_1 :=
  \K^{n_1}, \ldots V_r:= \K^{n_r}$. Since the category
  $(\Mod{A}{})^{\fg}$ is semisimple, we may decompose the
  finitely-generated $A$-modules $M$ and $N$ as the direct sum of
  simple modules:
  \begin{align}
    \begin{aligned}
      M &\cong \left(  \bigoplus_{i_1=1}^{l_1} \K^{n_1}  \right) \oplus  \left(  \bigoplus_{i_2=1}^{l_2} \K^{n_2}  \right) \oplus \cdots \oplus  \left(  \bigoplus_{i_r=1}^{l_r} \K^{n_r}  \right) \\
      N & \cong \left( \bigoplus_{i_1=1}^{l_1'} \K^{n_1} \right)
      \oplus \left( \bigoplus_{i_2=1}^{l_2'} \K^{n_2} \right) \oplus
      \cdots \oplus \left( \bigoplus_{i_r=1}^{l_r'} \K^{n_r} \right).
    \end{aligned}
  \end{align}
  By Schur's lemma, any $f \in \Hom_A(M,N)$ is given by $f=f_1 \oplus
  f_2 \oplus \ldots \oplus f_r$ where $f_i $ is a $l_i' \times
  l_i$-matrix. Similarly, any $g \in \Hom_A(N,M)$ is given by $g=g_1
  \oplus g_2 \oplus \ldots \oplus g_r$ where each $g_i$ is a $l_i
  \times l_i'$ matrix.  Thus,
  \begin{equation}
    \begin{aligned}
      \tr^\lambda_M(g \circ f)& = \tr^\lambda_{M}((g_1 f_1) \oplus (g_2 f_2) \oplus \ldots \oplus g_r f_r) \\
      &= \sum_{i=1}^r \tr^\lambda_{ (\K^{n_i})^{l_i}}(g_i f_i)  && \text{(by additivity)} \\
      &= \sum_{i=1}^r \sum_{j=1}^{l_i} \tr^\lambda_{\K^{n_i}}( (g_i f_i)_{j,j}) && \text{(by lemma \ref{lem:trace-matrix-form})} \\
      &=\sum_{i=1}^r \sum_{j=1}^{l_i} \sum_{k=1}^{l_i'}
      \tr^\lambda_{\K^{n_i}} \left((g_i)_{j,k} \circ (f_i)_{k,j}
      \right)
    \end{aligned}
  \end{equation}
  Since $f$ was assumed to be non-zero, at least one $(f_i)_{j,k}$ is
  non-zero. Suppose that $(f_{\tilde i})_{\tilde j, \tilde k}
  \in \End_A(\K^{n_i})$ is not the zero morphism. By Schur's lemma,
  $(f_{\tilde i})_{\tilde j, \tilde k}$ is an isomorphism. Now define
  $g \in \Hom_A(N,M)$ as
  \begin{equation}
    (g_i)_{j,k}:= (\lambda_{\tilde i})^{-1} \delta_{i, \tilde i} \delta_{j, \tilde j} \delta_{k, \tilde k} (f_{\tilde i})_{\tilde k, \tilde j}^{-1}.
  \end{equation}
  Then, by example \ref{ex:rank-of-k^n},
  \begin{equation}
    \begin{aligned}
      \tr^\lambda_M(g \circ f) = (\lambda_{\tilde
        i})^{-1}\tr^\lambda_{\K^{n_{\tilde i}}}(\id_{\K^{n_{\tilde
            i}}}) = 1_\K \neq 0.
    \end{aligned}
  \end{equation}
\end{proof}

We summarize the situation with the following proposition:
\begin{prop}
  \label{prop:a-mod-is-cy}
  Let $(A, \lambda)$ be a semisimple symmetric Frobenius
  algebra. Then, the category of finitely-generated $A$-modules
  $(\Mod{A}{})^{\fg}$ has got the structure of a Calabi-Yau category
  with trace $\tr_M^\lambda: \End_A(M) \to \K$ as defined in equation
  \eqref{eq:def-trace}.
\end{prop}
\begin{proof}
  It is well-known that the category of finite-dimensional modules
  over a finite-dimensional algebra is a finite, linear category,
  cf. \cite{dss12}.
  
  If $M$ is a finitely-generated $A$-module, the trace
  $\tr^\lambda(M): \End(M) \to \K$ as defined in equation
  \eqref{eq:def-trace} is symmetric by lemma
  \ref{lem:trace-symmetric}, while the induced bilinear form is
  non-degenerate by lemma \ref{trace-non-deg}.  This shows that
  $(\Mod{A}{})^{\fg}$ is a Calabi-Yau category.
\end{proof}

The following example shows that assumption that $A$ is semisimple is
a necessary condition.
\begin{ex}[Counter-example]
Let $\K$ be a field of characteristic two, and consider the group
  algebra $A:=\K[\Z_2]$. Then, $A \cong \K[x]/( x^2-1) \cong
  \K[x]/(x-1)^2$. This is in fact a Frobenius algebra with Frobenius
  form $\lambda(g)=\delta_{g,e}$, which is not separable. Let $S$ be
  the trivial representation, and consider a projective
  two-dimensional representation of $A$ which we shall call $P$. Here,
  the non-trivial generator $g$ of $A$ acts on $P$ by the matrix
  \begin{equation}
    \qquad g =
    \begin{pmatrix}
      1                       & 1         \\
      0 & 1
    \end{pmatrix}.
  \end{equation}
  One easily computes that
  \begin{equation}
       \Hom(P,S)  \cong \left\{
      \begin{pmatrix}
        0 & b
      \end{pmatrix}
      \mid b \in \K \right\}, \qquad  ~\text{and}~             \qquad
    \Hom(S,P)  \cong \left\{
      \begin{pmatrix}
        a                                     \\
        0
      \end{pmatrix}
      \mid a \in \K \right\}.
  \end{equation}
  We claim that there is no trace on the representation category of
  $A$. Indeed, let $\tr_S : \End(S) \to \K$ be any linear map. Then,
  the pairing
  \begin{align}
    \Hom(S,P) \ot \Hom(P,S) & \to \K    \\
    \begin{pmatrix}
      a \\0
    \end{pmatrix}
    \ot
    \begin{pmatrix}
      0 & b
    \end{pmatrix}
    & \mapsto \tr_S\left(
      \begin{pmatrix}
        0 & b
      \end{pmatrix}
      \begin{pmatrix}
        a \\0
      \end{pmatrix}
    \right) =0
  \end{align}
  is always degenerate. Therefore, a non-degenerate pairing does not
  exist.
\end{ex}
\subsection{Constructing the 2-functor \texorpdfstring{$\Rep^{\fg}$}{Rep^{fg}} on 1-morphisms} 
The next step of the construction will be the value of $\Rep^{\fg}$ on
1-morphisms of $\Frob$, which are compatible Morita contexts. To
these, we will have to assign equivalences of Calabi-Yau categories. Let us
recall a classical theorem from Morita theory:
\begin{theorem}[{\cite[Theorem 3.4 and 3.5]{bass}}]
  Let $A$ and $B$ be $R$-algebras, and let $({_BM_A},{_AN_B},\eps,
  \eta)$ be a Morita context between $A$ and $B$. Then,
  \begin{enumerate}
  \item $M$ and $N$ are both finitely-generated and projective as
    $B$-modules.
  \item An $A$-module $X$ is finitely generated over $A$ if and only
    if $M \ot_A X$ is finitely generated over $B$.
  \item The functor
    \begin{align}
      M \ot_A - : \Mod{A}{} \to \Mod{B}{}
    \end{align}
    is an equivalence of linear categories.
  \end{enumerate}
\end{theorem}
This theorem suggests that we should define $\Rep^{\fg}$ on Morita
contexts by the functor which tensors with the bimodule $M$. In
order for this to be well-defined, this functor should be a Calabi-Yau
functor as in definition \ref{def:calabi-yau-functor} if the Morita
context is compatible with the Frobenius forms as in definition \ref{def:morita-compatible}. In order to
show this, we need an additional lemma:
\begin{lemma}
  \label{lem:morita-context-psi-tensor}
  Let $A$ and $B$ be two semisimple $\K$-algebras.  Let
  $\cM=(M,N,\eps,\eta)$ be a Morita context between $A$ and $B$. Write
  \begin{equation}
    \label{eq:eps-inverse}
    \eps^{-1}(1_A)=\sum_{i,j}n_i \ot m_j \in N \ot_B M.
  \end{equation}
For $T$ a finitely-generated left $A$-module, define a linear map
  \begin{equation}
    \begin{aligned}
      \xi:  T^* \ot_A T &\to (M \ot_A T)^* \ot_B (M \ot_A T) \\
      t^* \ot t & \mapsto \left( \left(x \ot y \mapsto \sum_i
          \eta^{-1} \left(x.t^*(y) \ot n_i \right) \right) \ot \sum_j
        m_j \ot t \right)
    \end{aligned}
  \end{equation}
  Then, the following diagram commutes.
  \begin{equation}
    \begin{tikzcd}[row sep=large]
      T^*\ot_AT \dar[swap]{\Psi_{T,T}} \rar{\xi} & (M \ot_A T)^* \ot_B
      (M \ot_A T) \dar{\Psi_{M \ot_A T, M \ot_A T}} \\ \End_A(T)
      \rar[swap]{\id_M \ot - } &\End_B(M \ot_A T)
    \end{tikzcd}
  \end{equation}
\end{lemma}
\begin{proof}
  First note that
  \begin{equation}
    \label{eq:morita-context-condition}
    \sum_{i,j} \eta^{-1}(x \ot n_i).m_j =x 
  \end{equation}
  for every $x$ in $M$, since $\eps$ and $\eta$ are part of a Morita
  context.  Now, we calculate:

  \begin{align}
    \begin{aligned}
      ( \id_M\ot- \circ \Psi_{T,T})(t^* \ot t)(x \ot y)= (\id_M \ot
      -)(y \mapsto t^*(y).t)(x \ot y)= x \ot t^*(y).t
    \end{aligned}
  \end{align}
  On the other hand,
  \begin{equation}
    \begin{aligned}
      & ( \Psi_{M \ot_A T, M \ot_A T} \circ \xi) (t^* \ot t)(x \ot y)=
      \\ & = (\Psi_{M \ot_A T, M \ot_A T})\left( \left(x \ot y \mapsto
          \sum_i \eta^{-1} \left(x.t^*(y) \ot n_i \right) \right) \ot
        \sum_j m_j \ot t \right)(x \ot y) \\ &=
      \sum_{i,j }\eta^{-1}(x.t^*(y) \ot n_i) .(m_j \ot t) \\
      &= \sum_{i,j}\eta^{-1}(x.t^*(y) \ot n_i). m_j \ot t \\
      &=x.t^*(y) \ot t
    \end{aligned}
  \end{equation}
  where in the last line, we have used equation
  \eqref{eq:morita-context-condition}.  This shows that the diagram
  commutes.
\end{proof}
The next proposition shows how the compatibility condition on the
Morita context between two Frobenius algebras in definition \ref{def:morita-compatible} is equivalent
to the fact that tensoring with the bimodule $M$ of the Morita context
is a Calabi-Yau functor:
\begin{prop}
  \label{prop:compatible-morita-cy-functor}
  Let $(A,\lambda^A)$ and $(B, \lambda^B)$ be two semisimple symmetric
  Frobenius algebras, and let $(M,N,\eps, \eta)$ be a Morita context
  between $A$ and $B$. Endow $\Rep^{\fg}(A)$ and $\Rep^{\fg}(B)$ with
  the Calabi-Yau structure as in definition
  \ref{cy-structure-rep}. Then, the Morita context is compatible with
  the Frobenius forms $\lambda^A$ and $\lambda^B$ as in definition
  \ref{def:morita-compatible} if and only if
  \begin{equation}
    (M \ot_A -) : \Rep(A)^{\fg} \to \Rep(B)^{\fg}   
  \end{equation}
  is a Calabi-Yau functor as in definition
  \ref{def:calabi-yau-functor}.
\end{prop}


\begin{proof}
  Let $_AT$ be a finitely-generated left $A$-module.  By definition,
  the functor $M \ot_A-$ is a Calabi-Yau functor if and only if
  \begin{equation}
    \tr_{M \ot_A T}^{\lambda^B} (\id_M \ot f)= \tr_T^{\lambda^A}(f)
  \end{equation}
  for all $f \in \End_A(T)$. We have to calculate the left hand-side:
  Let $f \in \End_A (T)$ and write
  \begin{equation}
    \Psi^{-1}_{T,T}(f)=\sum_{i,j}t_i^* \ot t_j \in T^* \ot_A T .
  \end{equation}
  Using $n_i$ and $m_j$ as introduced in formula
  \eqref{eq:eps-inverse}, lemma
  \ref{lem:morita-context-psi-tensor} shows that
  \begin{equation}
    \begin{aligned}
      \Psi_{M \ot_A T, M \ot_A T}^{-1}( \id_M \ot f)&= \xi \circ \Psi^{-1}_{T,T}(f)\\ &=\sum_{i,j}\xi(t_i^* \ot t_j)\\
      &= \left( x \ot y \mapsto \sum_{k,i} \eta^{-1} \left(x.t_i^*(y)
          \ot n_k \right) \right) \ot \sum_{l,j} m_l \ot t_j .
    \end{aligned}
  \end{equation}
  Hence,
  \begin{equation}
    \label{eq:2}
    \begin{aligned}
      \tr_{M \ot_A T}^{\lambda^B}(\id_M \ot f)&=({\lambda^B} \circ \ev \circ \Psi_{M \ot_A T, M \ot_A T}^{-1})(\id_M \ot f) \\
      &=\sum_{i,j,k,l} {\lambda^B}(\eta^{-1}(m_l.t_i^*(t_j) \ot n_k)).
    \end{aligned}
  \end{equation}
  Since
  \begin{equation}
    \label{eq:3}
    \tr^{\lambda^A}_T(f)= \sum_{i,j} \lambda^A (t_i^*(t_j))  ,
  \end{equation}
  the functor $M \ot_A -$ is a Calabi-Yau functor if and only if the
  right hand sides of equations \eqref{eq:2} and \eqref{eq:3} agree for
  every $t_i \in T^*$ and $t_j \in T$. Using the fact that the
  Frobenius forms are symmetric, and thus factor through $A/[A,A]$,
  this is the case if and only if the Morita context is compatible
  with the Frobenius forms as in equation
  \eqref{eq:morita-compatible-form}.
\end{proof}

\begin{newdef}
  Proposition \ref{prop:compatible-morita-cy-functor} enables us to
  define the 2-functor $\Rep^{\fg}$ on 1-morphisms of the bicategory $\Frob$: we assign
  to a compatible Morita context $\cM:=(M,N, \eps, \eta)$ between two
  semisimple symmetric Frobenius algebras $A$ and $B$ the equivalence
  of Calabi-Yau categories $\Rep^{\fg}(\cM)$ given by
  \begin{equation}
    \Rep^{\fg}(\cM) :=( M \ot_A -) : \Rep^{\fg}(A) \to \Rep^{\fg}(B).
  \end{equation}
\end{newdef}

\subsection{Constructing the 2-functor \texorpdfstring{$\Rep^{\fg}$}{Rep^{fg}} on 2-morphisms}
Let $(M,N,\eps, \eta)$ and $(M', N',\eps', \eta')$ be two compatible Morita
contexts between semisimple symmetric Frobenius algebras $A$ and $B$, and let $\alpha: M \to M'$
and $\beta: N \to N'$ be a morphism of Morita contexts. We define a
natural transformation $\Rep^{\fg}((\alpha, \beta)):(M \ot_A -) \to( M' \ot_A -)$ as follows: for
every left $A$-module $_AX$, we define the component of the natural
transformation as
\begin{equation}
 \Rep^{\fg}((\alpha, \beta))_X := (\alpha \ot \id_X) : M \ot_A X \to M' \ot_A X.
\end{equation}
This is indeed a natural transformation because for every morphism $f:
{_AX} \to{ _AY}$ of left $A$-modules, the following diagram
\begin{equation}
  \begin{tikzcd}[column sep=large, row sep=large]
    M \ot_A X \rar{\alpha \ot \id_X} \dar[swap]{\id_M \ot f} & M'
    \ot_A X
    \dar{\id_{M'} \ot f} \\
    M \ot_A Y \rar[swap]{\alpha \ot \id_Y} & M' \ot_A Y
  \end{tikzcd}
\end{equation}
commutes.

Thus, we have obtained the following weak 2-functor $\Rep^{\fg}$:
\begin{equation}
  \begin{aligned}
    \Rep^{\fg}: \Frob  & \to \CY \\
    (A, \lambda^A) &\mapsto \left(\Rep^{\fg}(A), \tr^{\lambda^A} \right) \\
    ({_BM_A}, {_BN_A}, \eps, \eta) & \mapsto \left(M \ot_A - :
      (\Rep^{\fg}(A),
      \tr^{\lambda^A}) \to (\Rep^{\fg}(B),  \tr^{\lambda^B})  \right) \\
    \left( (\alpha, \beta): ({_BM_A}, {_BN_A}, \eps, \eta) \to
      ({_BM'_A}, {_BN'_A}, \eps', \eta') \right) & \mapsto \left(
      \alpha \ot \id_{ -} :( M \ot_A -) \to( M' \ot_A -) \right)
  \end{aligned}
\end{equation}
Observe that by the definition of the Deligne tensor product, this
weak 2-functor is compatible with the symmetric monoidal structures of
$\Frob$ and $\CY$, and thus is a symmetric monoidal 2-functor.
\section{Proving the equivalence}
\label{sec:proving-equivalence}
The aim of this section is to prove that the weak 2-functor
$\Rep^{\fg}: \Frob \to \CY$ constructed in section
\ref{sec:rep-construction} is an equivalence of bicategories. This
will be done in several steps. First, we show that $\Rep^{\fg}$ is
essentially surjective. Let $(\cC , \tr^\cC)$ be a Calabi-Yau
category, and let $X_1, \ldots, X_n$ be representatives of the
isomorphism classes of simple objects of $\cC$. Define an object
$P$ of $\cC$ as $P:=\oplus_{i=1}^n X_i$. Clearly, $A:= \End_\cC(P)$ is
a semisimple, symmetric Frobenius algebra over $\K$ with Frobenius
form $\lambda$ given by $\lambda:=\tr^\cC_P$. By proposition
\ref{prop:a-mod-is-cy}, the category $(\Mod{A}{})^{\fg}$ has the
structure of a Calabi-Yau category. We now claim:
\begin{prop}
  \label{prop:calabi-yau-equivalent-a-mod}
  The functor
  \begin{equation}
    \Hom_\cC(P, -):\cC \to (\Mod{A}{})^{\fg}
  \end{equation}
  is an equivalence of Calabi-Yau categories.
\end{prop}
\begin{proof}
  An exercise of \cite{eghlsvy11} which is proven in \cite[Proposition
  1.4]{dss12} asserts that the functor $\Hom_\cC(P,-)$ is an
  equivalence of linear categories. Thus, our claim amounts to showing
  that this functor is compatible with the traces as required in
  definition \ref{def:calabi-yau-functor}.

  Write an object $X$ of $\cC$ as an arbitrary sum of simple objects,
  so that $X= \oplus _{j =1}^m X_j$, and let $f
  \in \End_\cC(X)$. Since $\cC$ is an additive category, we can represent $f$
  as an $m \times m$ matrix
  \begin{equation}
    f=
    \begin{pmatrix}
      f_{1, 1} & \ldots & f_{1, m} \\
      \vdots & & \vdots \\
      f_{m, 1} & \ldots & f_{m, m}
    \end{pmatrix}
  \end{equation}
  where $f_{k, l} \in \Hom_\cC(X_{l}, X_{k})$.

  Similarly, any $g \in \Hom_\cC(P,X) $ is naturally a $m \times n$
  matrix with entries
  \begin{equation}
    g=
    \begin{pmatrix}
      g_{1, 1} & \ldots & g_{1, m} \\
      \vdots & & \vdots \\
      g_{n, 1} & \ldots &g_{n, m}
    \end{pmatrix}
  \end{equation}
  where $g_{i, k} \in \Hom_\cC(X_{k}, X_i)$

  Under this identification, $A=\End_\cC(P)$ acts on $\Hom_\cC(P,X)$
  as $a.f:=f \cdot a$ where $f \cdot a$ is the matrix product of $f$
  and $a$.

  Then the morphism $\Hom_\cC(P,f)$ is given by
  \begin{equation}
    \begin{aligned}
      \Hom_\cC(P,f) :  \Hom_\cC(P,X) &\to \Hom_\cC(P,X) \\
      g &\mapsto f \cdot g
    \end{aligned}
  \end{equation}
  where $f \cdot g$ is the matrix product of $f$ and $g$. 
  As a first step to calculate the trace in $(\Mod{A}{})^{\fg}$, we
  claim that
  \begin{equation}
    \label{eq:psi-delta-tilde}
    \Psi^{-1}_{\Hom_\cC(P,X), \Hom_\cC(P,X)}( \Hom_\cC(P,f))= \delta^* \ot  \tilde{f},
  \end{equation}
  as an element of $ \Hom_\cC(P,X)^* \ot_A \Hom_\cC(P,X)$, where
  $\delta^* \in \Hom_\cC(P,X)^*$ and $ \tilde f \in \Hom_\cC(P,X)$ are
  defined as follows.  First, define the $m \times n$-matrix
  \begin{equation}
    \tilde f_{k, r} :=
    \begin{cases}
      f_{k, r} & \text{if $r \leq m$, } \\
      0 & \text{else}
    \end{cases}
  \end{equation}

  Now, given a $m \times n$ matrix $g\in \Hom_\cC(P,X)$, the element $
  \delta^*(g) $ of $A$ is defined to be an $n \times n$ matrix with
  entries
  \begin{equation}
    \delta^*(g)_{r,l} := 
    \begin{cases}
      g_{r, l} & \text{if $r\leq m$,} \\
      0 & \text{else}
    \end{cases}
  \end{equation}

  Then, if $g \in \Hom_\cC(P,X)$,
  \begin{equation}
    \begin{aligned}
      (\Psi((\delta)^* \ot \tilde f)(g))_{k,l} &=
      (\delta^*(g). \tilde f)_{k,l} \\
      &=(\tilde f \cdot \delta^*(g))_{k,l} \\
      &= \sum_{r=1}^n \tilde f_{k,r} \cdot \delta^*(g)_{r,l} \\
      &=  \sum_{r=1}^m f_{k, r}  \circ g_{r,l} \\
      & =(\Hom_\cC(P,f)(g))_{k,l}
    \end{aligned}
  \end{equation}
  This shows equation \eqref{eq:psi-delta-tilde}.

  We may now calculate the trace of the morphism $\Hom_\cC(P,f)$ in
  $(\Mod{A}{})^{\fg}$. By definition of the trace in $(\Mod{A}{})^{\fg}$,
  we have
  \begin{align}
    \tr^\lambda_{\Hom_\cC(P,X)}( \Hom_\cC(P,f)) = ( \tr_P \circ \ev
    \circ \Psi^{-1}_{\Hom_\cC(P,X), \Hom_\cC(P,X)})( \Hom_\cC(P,f)).
    \intertext{Then,}
    \begin{aligned}
      \tr^\lambda_{\Hom_\cC(P,X)}( \Hom_\cC(P,f)) &=
      (\tr_P \circ \ev)(\delta^* \ot \tilde f) \qquad  &&  \text{(by equation \eqref{eq:psi-delta-tilde})} \\
      &=\tr_P(\delta^*( \tilde f)) \\
      &= \sum_{i=1}^n \delta^*(\tilde f)_{ii} && \text{(by lemma \ref{lem:trace-matrix-form})} \\
      &=\sum_{i=1}^m f_{i, i } \\
      &= \tr_X^\cC(f). && \text{(by lemma
        \ref{lem:trace-matrix-form})}
    \end{aligned}
  \end{align}
  This shows that $\Hom_\cC(P,-)$ is a Calabi-Yau functor.
\end{proof}

Next, we follow the exposition in \cite[Proposition 3.1]{bass} and show
that the functor $\Rep^{\fg}$ is essentially surjective on
1-morphisms. In detail:
\begin{prop}
  \label{prop:rep-ess-surj-1-mor}
  Let $(A,\lambda^A)$ and $(B,{\lambda^B})$ be two semisimple,
  symmetric Frobenius algebras. Endow $(\Mod{A}{A})^{\fg}$ and
  $(\Mod{B}{})^{\fg}$ with the Calabi-Yau structure described in
  proposition \ref{prop:a-mod-is-cy}, and let
  \begin{equation}
    F:(\Mod{A}{})^{\fg} \rightleftarrows (\Mod{B}{})^{\fg} :G
  \end{equation}
  be an equivalence of Calabi-Yau categories.

  Then, there is a compatible Morita context $\cM$ between $A$ and
  $B$, so that $\Rep^{\fg}(\cM) \cong F$.
\end{prop}
\begin{proof}
The proof works essentially by an application of Eilenberg-Watts and
by checking that everything is compatible with the traces:
  define a $(B,A)$-bimodule $M$ as $M:=F(A)$ which is naturally a left
  $B$-module, and a right $A$-module by using the map
  \begin{equation}
    A \cong \End_A(A) \xrightarrow{F} \End_B(M).
  \end{equation}
  The Eilenberg-Watts theorem then shows that the functor $F$ is
  naturally isomorphic to $M \ot_A -$ (cf. theorem 1 in
  \cite{watts-theorem}). Thus $F$ is a Calabi-Yau functor if and only
  if $M \ot_A - $ is a Calabi-Yau functor.  Similarly, there is an
  $(A,B)$-bimodule $N$ given by $N:=G(B)$, so that the functor $G$ is
  naturally isomorphic to $N \ot_B-$.

  Furthermore, there are isomorphisms of bimodules
  \begin{equation}
    \begin{aligned}
      \eps: N \ot_B M \cong G(M) \cong G(M \ot_A A ) \cong G(F(A)) \cong A \\
      \eta: B \cong F(G(B)) \cong F(N \ot_B B ) \cong F(N) \cong M
      \ot_A N
    \end{aligned}
  \end{equation}
  since $F$ and $G$ is an equivalence of categories.  We claim that we
  can choose these isomorphisms in such a way that $(M,N,\eps,\eta)$
  becomes a Morita context.
 
  Indeed, by lemma \ref{lem:morita-context-only-1-diagram} it suffices
  to show that diagram \eqref{eq:def-morita-context-2} commutes. Let
  $r_B : N \ot_B B \to N$ be right-multiplication, and let $l_A: A
  \ot_B N \to N$ be left-multiplication.

  Since $\eps$, $\eta$, $r_B$ and $l_A$ are isomorphisms of bimodules,
  there is a $u \in \Aut_{(A,B)}(N)$, so that
  \begin{equation}
    r_B \circ \id_N \ot \eta^{-1} = u \circ l_A \circ \eps \ot \id_N.
  \end{equation}
  In particular,
  \begin{equation}
    u \in \Hom_A(N,N) \cong \Hom_A(G(B), G(B)) \cong \Hom_B(B,B). 
  \end{equation}
  Since every morphism of left $B$-modules $u \in \Hom_B(B,B)$ is
  given by right multiplication with an element of $B$, we may
  identify $u$ with this element.  Since $u$ is also a morphism of
  right $B$-modules, the element $u$ is in the center of $B$.

  Now define an isomorphism of $(B,B)$-bimodules
  \begin{equation}
    \begin{aligned}
      \tilde \eta^{-1} : M \ot_A N &\to  B \\
      m \ot n &\mapsto u.\eta^{-1}(n \ot m).
    \end{aligned}
  \end{equation}
  Now, if we replace $\eta$ by $\tilde \eta$ we have made diagram
  \eqref{eq:def-morita-context-1} commute. Thus, $(M,N, \eps, \tilde
  \eta )$ is a Morita context, which is compatible with the Frobenius
  forms by proposition \ref{prop:compatible-morita-cy-functor}.
\end{proof}
We are now ready to prove the following theorem:
\begin{theorem}
  \label{thm:cy-frob-equivalence}
  The weak 2-functor $   \Rep^{\fg}: \Frob \to \CY $ is an equivalence
  of bigroupoids.
 \end{theorem}
\begin{proof}
  In a first step, proposition \ref{prop:a-mod-is-cy} shows the
  representation category of a semisimple symmetric Frobenius algebra
  has indeed got the structure of a Calabi-Yau
  category. Furthermore, this assignment is essentially surjective on
  objects: given a Calabi-Yau category $\cC$,
  proposition \ref{prop:calabi-yau-equivalent-a-mod} shows how to
  construct a Frobenius algebra $A$ so that $\Rep^{\fg}(A)$ and $\cC$
  are equivalent as Calabi-Yau categories.

  Now, given a compatible Morita context between two semisimple
  symmetric Frobenius
  algebras $A$ and $B$, proposition
  \ref{prop:compatible-morita-cy-functor} shows how to construct a
  Calabi-Yau functor between the representation
  categories. Furthermore, this assignment is essentially surjective
  by proposition \ref{prop:rep-ess-surj-1-mor}.

  Finally, one shows by hand that $\Rep^{\fg}$ induces a bijection on
  2-morphisms of $\Frob$ and $\CY$ which carry no
  additional structures or properties.
\end{proof}
\begin{remark}
  Note that both $\Frob$ and $\CY$ have an additional symmetric
  monoidal structure, which is preserved by the 2-functor
  $\Rep^\fg$. Essentially, this follows from the definition of the
  Deligne tensor product, and by observing that definition
  \ref{cy-structure-rep} of the Calabi-Yau structure on the
  representation category of a semisimple symmetric Frobenius algebra
  is well-behaved under taking tensor products. Thus, $\Rep^\fg: \Frob \to
  \CY$ is even an equivalence of symmetric monoidal bigroupoids.
\end{remark}

\begin{remark}
  Let us also comment on the relationship between the weak 2-functor
  $\Rep^\fg:\Frob \to \CY$ in theorem \ref{thm:cy-frob-equivalence}
  and the homotopy fixed points of a $G$-action on a
  bicategory as considered in \cite{hsv16}: by \cite[Corollary 4.2 and 4.5]{hsv16}
  there are equivalences of bigroupoids
\begin{equation}
  \begin{aligned}
    \core{\Alg_2^\fd}^{SO(2)} & \cong \Frob \\
\core{\Vect_2^\fd}^{SO(2)} & \cong \CY
  \end{aligned}
\end{equation}
where $\core{\Alg_2^\fd}^{SO(2)}$ and $\core{\Vect_2^\fd}^{SO(2)}$ are
the bigroupoids of homotopy fixed points of the trivial $SO(2)$-action
on the core of fully-dualizable objects of $\Alg_2$ and
$\Vect_2$. Now, as the $SO(2)$-action is trivial, the weak 2-functor
$\Rep:\core{\Alg_2^\fd} \to \core{\Vect_2^\fd}$ sending a semisimple
algebra to its category of representations induces a weak 2-functor on
homotopy fixed points $\Rep^{SO(2)}:\core{\Alg_2^\fd}^{SO(2)} \to
\core{\Vect_2^\fd}^{SO(2)}$. One now checks that this induced
2-functor on homotopy fixed points is pseudo-naturally isomorphic to
the weak 2-functor $\Rep^\fg$ of theorem
\ref{thm:cy-frob-equivalence}. More precisely, the following diagram
commutes up to pseudo-natural isomorphism:
\begin{equation}
 \vcenter{\hbox{\includegraphics{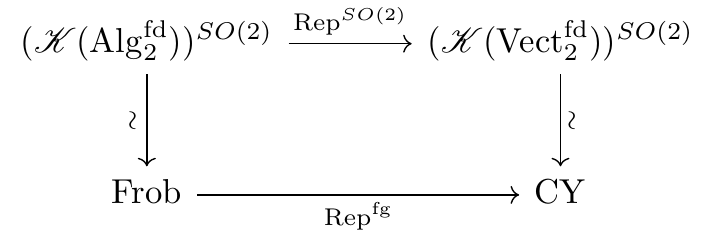}}}
\end{equation}
The commutativity of the diagram boils down to a Schur's lemma type of
argument with a vector of invertible scalars after following the
construction in \cite{hsv16}. Details of the construction are given in the
appendix and in \cite{hessephd}.
\end{remark}
\appendix
\section{Induced functors on homotopy fixed points}
\label{sec:induced-functors}
So far, we have constructed an equivalence of bicategories $\Rep^{\fg}:
\Frob \to \CY$. In this section, we show that this equivalence is
actually the \enquote{equivariantization} of the 2-functor sending an
algebra to its category of modules. 

In order to do so, we introduce the notion of a the
\enquote{equivariantization} of a weak
2-functor between bicategories equipped with a $G$-action, where $G$
is a topological group. Let us briefly recall the relevant
definitions: For a group $G$, we denote with $\B G$ the category with
one object and $G$ as morphisms. Similarly, if $\cC$ is a monoidal
(bi-)category, $\B\cC$ will denote the (tri-)bicategory with one object and $\cC$
as endomorphism (bi-)category of this object.  

For a topological group $G$, let $\Pi_2(G)$ be its fundamental
2-groupoid, and $\B \Pi_2(G)$ the tricategory with one object called
$*$ and $\Pi_2(G)$ as endomorphism bicategory. A $G$-action on a
bicategory $\cC$ is then defined to be a trifunctor $\rho:\B \Pi_2(G)
\to \Bicat$ with $\rho(*)=\cC$, where $\Bicat$ is the tricategory of
bicategories. Furthermore, given a $G$-action $\rho$ on a bicategory,
we define the bicategory of homotopy fixed points $\cC^G$ to be the
bicategory $\Nat(\Delta, \rho)$ where objects are given by tritransformations
between the constant functor $\Delta$ and $\rho$, 1-morphisms are
modifications, and 2-morphisms are 
perturbations. In the following, we will use the notation of
\cite{hsv16} concerning homotopy fixed points.
\begin{newdef}
  \label{def:g-equivariant-structure}
Let $\rho :\B\Pi_2(G) \to \Bicat$ be a $G$-action on the bicategory
  $\rho(*)=\cC$, and let $\rho':\B\Pi_2(G) \to \Bicat$ be a
  $G$-action on $\rho'(*)=\cD$. Let $H: \cC \to \cD$ be a weak
  2-functor. A $G$-equivariant structure for the weak 2-functor $H: \cC
  \to \cD$ consists of:
  \begin{itemize}
  \item a pseudo-natural transformation $T$ in the diagram
    \begin{equation}
      \vcenter{\hbox{\includegraphics{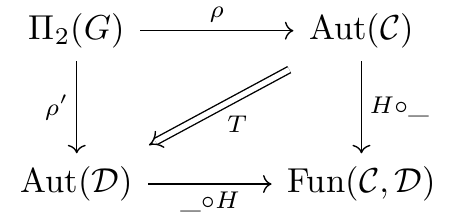}}}
    \end{equation}
    which very explicitly consists of the data:
    \begin{itemize}
    \item a pseudo-natural transformation
      \begin{equation}
        T_g : H \circ F_g \to F_g' \circ H
      \end{equation}
      for every $g \in G$, explaining the name $G$-equivariant
      structure, 
    \item For every path $\gamma: g \to h$, an invertible modification
      $T_\gamma$ in the diagram
      \begin{equation}
        \vcenter{\hbox{\includegraphics{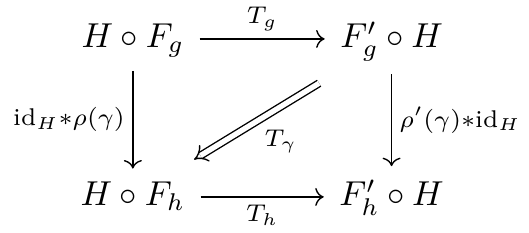}}}
      \end{equation}
    \end{itemize}
  \item for every $g$, $h \in G$, invertible modifications $P_{gh}$
    \begin{equation}
      \vcenter{\hbox{\includegraphics{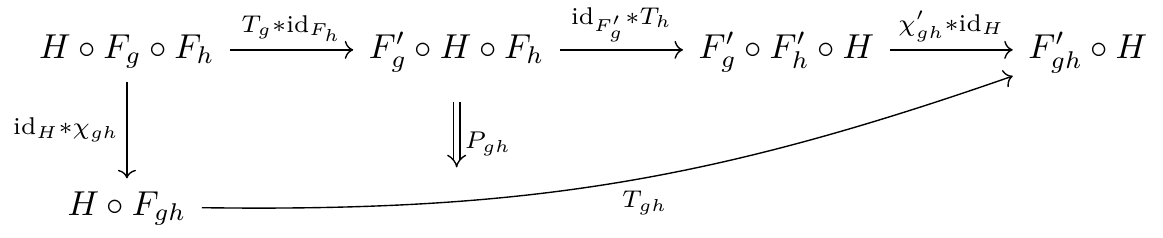}}}
    \end{equation}
  \item a modification $N$
    \begin{equation}
      \vcenter{\hbox{\includegraphics{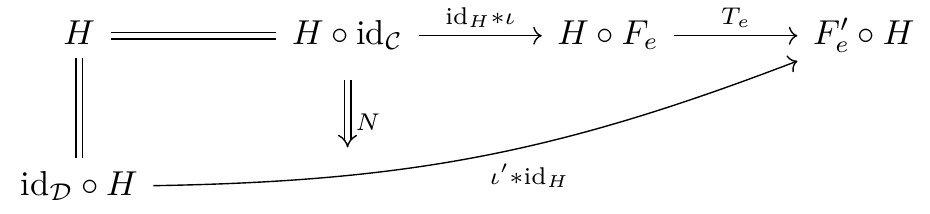}}}
    \end{equation}
     \end{itemize}
  so that the three equations of a tritransformation in definition 3.3
  of \cite{gps95} are fulfilled.
\end{newdef}
\begin{remark}
  We have defined a $G$-equivariant structure on a weak 2-functor $H$
  in such a way that it induces a tritransformation $\rho \to
  \rho'$ between the two actions. It is crucial to remark that the $G$-equivariant structure
  induces a weak 2-functor $H^G$ on homotopy fixed point bicategories:
  \begin{equation}
    H^G:\cC^G = \Nat(\Delta, \rho) \to \Nat(\Delta, \rho') =\cD^G.
  \end{equation}
\end{remark}
Explicitly, the induced functor on homotopy fixed points is given as follows:
\begin{newdef}
  \label{def:induced-fixed-points}
  Suppose that $H:\cC \to \cD$ is a weak 2-functor between
  bicategories endowed with $G$-actions $\rho$ and $\rho'$, and
  suppose that $H$ possesses a $G$-equivariant structure as in
  definition \ref{def:g-equivariant-structure}. Then, the induced
  functor $H^G : \cC^G \to \cD^G$ is then given as follows: On objects
  $(c, \Theta, \Pi, M)$ as defined in \cite[Remark 3.11]{hsv16} of the
  homotopy fixed point bicategory $\cC^G$ we define:
  \begin{itemize}
  \item On the object $c$ of $\cC$, we have $H^G(c):=H(c)$,
  \item On the pseudo-natural equivalence $\Theta$, we define the
    functor on the 1-cell $\Theta_g: c \to F_g(c)$ by
    \begin{equation}
      H^G( \Theta_g):= \left( H(c) \xrightarrow{H(\Theta_g)} H(F_g(c)) \xrightarrow{T_g(c)} F_g'(H(c))  \right),
    \end{equation}
    where $F_g$ and $F'_g$ are data given by the action as defined
    in \cite[Remark 3.8]{hsv16}, whereas on the 2-dimensional
    component $\Theta_\gamma$ in the diagram
    \begin{equation}
      \vcenter{\hbox{\includegraphics{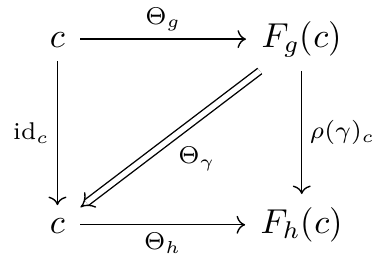}}}
    \end{equation}
    we assign the 2-morphism
    \begin{equation}H^G(\Theta_\gamma):= \qquad
      \vcenter{\hbox{\includegraphics{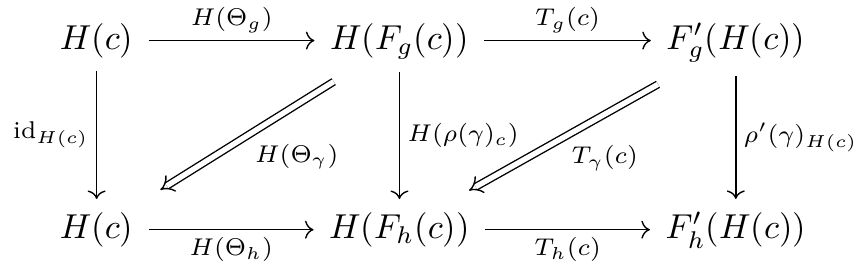}}}
    \end{equation}
  \item For the modification $\Pi$, we assign the 2-morphism
  \end{itemize}
  \begin{equation}
    \vcenter{\hbox{\includegraphics[width=\textwidth]{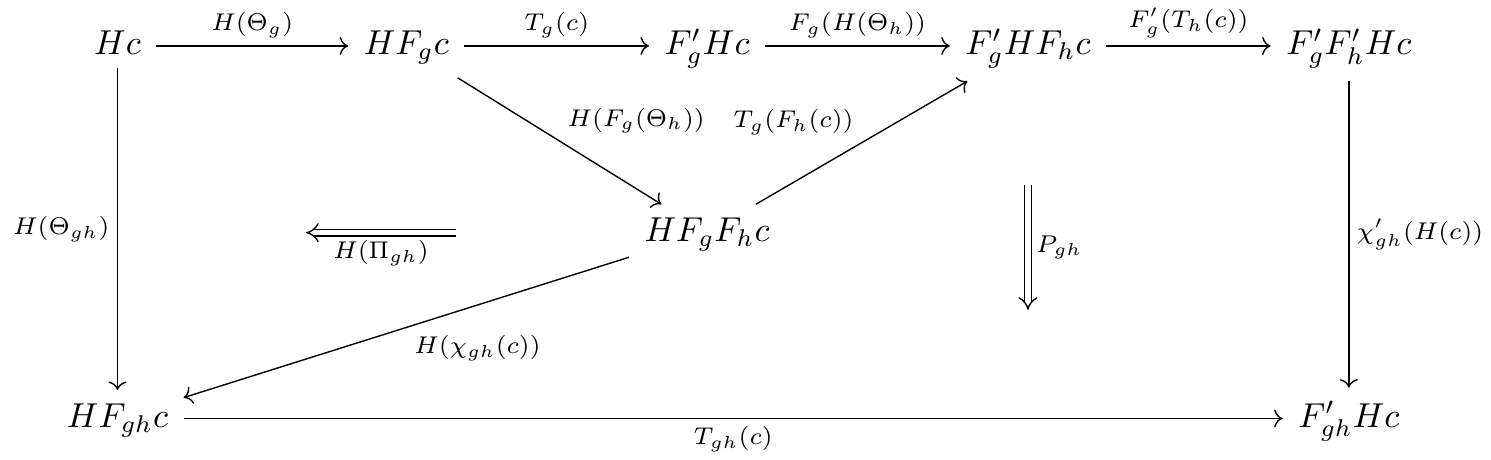}}}
  \end{equation}
  \begin{itemize}
  \item The induced functor sends the modification $M$ to the
    modification in the diagram
    \begin{equation}
      \vcenter{\hbox{\includegraphics{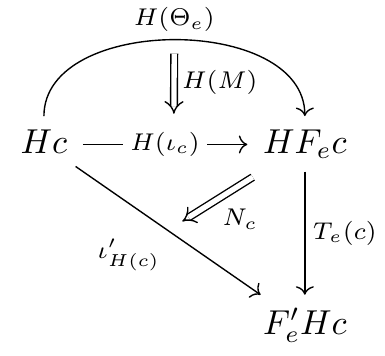}}}
    \end{equation}
  \end{itemize}
  It is a straightforward, but tedious to see that $(H^G(c),
  H^G(\Theta), H^G(\Pi), H^G(M))$ is a homotopy fixed point in
  $\cD^G$.
\end{newdef}
\begin{newdef}
  \label{def:induced-1-morphims}
  If $(f,m): (c, \Theta, \Pi, M) \to (\tilde c , \tilde \Theta ,
  \tilde \Pi, \tilde M)$ is a
  morphism of homotopy fixed points in $\cC^G$ as considered in
  \cite[Remark 3.12]{hsv16}, the induced functor
  $H^G$ is given on 1-morphisms of homotopy fixed points by $H^G(f):=
  H(f)$, and by
  \begin{equation}
    H^G(m_g) := \qquad 
    \vcenter{\hbox{\includegraphics{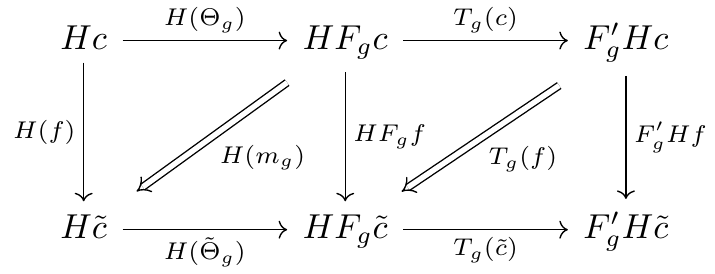}}}
  \end{equation}
\end{newdef}
\begin{newdef}
  \label{def:induced-2-morphism}
  If $(f,m)$ and $(\xi, n)$ are two 1-morphisms of homotopy fixed
  points in $\cC^G$, and $\sigma:((f,m) \to (\xi,n))$ is a 2-morphism
  of homotopy fixed points as in \cite[Remark 3.14]{hsv16}, the
  induced functor on 2-morphisms is given by $H^G(\sigma):=
  H(\sigma)$.
\end{newdef}
\begin{lemma}
  Let $\Rep: \core{ \Alg_2^\fd }\to \core{ \Vect_2^\fd}$ be the functor
  which send a finite-dimensional, semisimple algebra
 to  its category of finitely-generated modules,
  and let $\rho: \Pi_2(SO(2)) \to \Aut(\core{ \Alg_2^\fd})$ and $\rho':
  \Pi_2(SO(2)) \to \Aut( \core{ \Vect_2^\fd})$ be the trivial actions. Then,
  $\Rep$ has got a canonical $SO(2)$-equivariant structure given by
  taking identities everywhere.
\end{lemma}
\begin{proof}
  We need to provide the data in definition
  \ref{def:g-equivariant-structure}: Since both actions are trivial, we
  may choose $T_g : \Rep \to \Rep$ to be the identity pseudo-natural
  transformation for every $g \in G$, and $T_\gamma : \id_{\Rep} \circ
  T_g \to T_h \circ \id_{\Rep}$ to be the identity modification for
  every path $\gamma: g \to h$. Furthermore, we may also choose
  $P_{gh}$ and $N$ to be the identity modifications.
\end{proof}
Since the representation functor is $SO(2)$-equivariant, it induces a
functor on homotopy fixed point bicategories by definitions
\ref{def:induced-fixed-points}, \ref{def:induced-1-morphims} and
\ref{def:induced-2-morphism}. We claim:
\begin{theorem}
  \label{thm:cy-frob-induced}
  The diagram
  \begin{equation}
\label{eq:rep-induced}
    \vcenter{\hbox{\includegraphics{pics/pdf/motivation}}}
  \end{equation}
  commutes up to a pseudo-natural isomorphism. Here, the unlabeled
  equivalences are induced by \cite[Corollary 4.2 and 4.5]{hsv16},
  while the functor $ \Rep^{\fg}$ is constructed in section \ref{sec:rep-construction}.
\end{theorem}
\begin{proof}
 Let
 \begin{equation}
    \begin{aligned}
   F:   (\core{\Alg_2^\fd})^{SO(2)} & \cong \Frob \\
     G : (\core{\Vect_2^\fd})^{SO(2)} & \cong \CY
    \end{aligned}
  \end{equation}
  be the equivalences of bicategories constructed in \cite{hsv16}.  By
  \cite[Theorem 4.1]{hsv16}, the bicategory $(\core{
    \Alg_2^\fd})^{SO(2)}$ is equivalent to a bicategory where objects
  are given by semisimple algebras $A$, together with an isomorphism of
  Morita contexts $\lambda: \id_A \to \id_A$, where $\id_A$ is the
  identity Morita context, consisting of the algebra $A$ considered as
  an $(A,A)$-bimodule. If $(A,\lambda)$ is an object of $(\core{
    \Alg_2^\fd})^{SO(2)}$, we need to construct an equivalence of
  Calabi-Yau categories
\begin{equation}
\label{eq:eta-diagram}
  \eta_{(A,\lambda)} :( G \circ \Rep^{SO(2)})(A,\lambda) \to (\Rep^{\fg}
  \circ F) (A,\lambda).
\end{equation}
 By definition \ref{def:induced-fixed-points}, the value of $\Rep^{SO(2)}$ on $A$ is given
  by $\Rep^{\fg}(A)$, the category of finitely-generated modules of $A$. The value of $\Rep^{SO(2)}$ on $\lambda$
  is given by the natural isomorphism defined as follows: if $_AM$ is
  an $A$-module, the natural transformation $\Rep^{SO(2)}(\lambda)$ of
  the identity functor on $\Rep^{\fg}(A)$ is given
  in components by
  \begin{equation}
    \Rep^{SO(2)}(\lambda)_M:=   \left( M \cong A \ot_A M \xrightarrow{\lambda \ot
      \id_M} A \ot_A M \cong M \right).
  \end{equation}
We know that $A$ is isomorphic to a direct sum of matrix algebras: 
\begin{equation}
\label{eq:a-matrix}
  A  \cong \bigoplus_{i=1}^{r} M_{d_i}(\K).
\end{equation}
Let 
\begin{equation}
\label{eq:lambda-i}
(\lambda_1, \ldots, \lambda_r) \in \K^r \cong Z(A)
\cong \End_{(A,A)}(A)  
\end{equation}
 be the scalars corresponding to the isomorphism
of Morita contexts $\lambda$. Then, the Calabi-Yau structure on $(G
\circ \Rep^{SO(2)})(A, \lambda)$ is given as follows: it suffices to
write down a trace for the simple modules, because $\Rep^{\fg}(A)$ is
semisimple. If $X_i$ is a simple $A$-module, chasing through the
equivalence $G$ shows that the trace is given by identifying the
division algebra
$ \End_{\Rep^{\fg}(A)}(X_i)$ with the algebraically closed ground field $\K$ by Schur's
lemma, and then (up to a permutation of the simple modules)
multiplying with the scalar $\lambda_i$.

On the other hand, chasing through the equivalence of bicategories $F$
in \cite[Corollary 4.2]{hsv16}, we see that the Frobenius algebra
$F(A, \lambda)$ is given by the semisimple algebra $A$ as in equation \eqref{eq:a-matrix}, together with
the Frobenius form given by taking direct sums of matrix traces,
multiplied with the scalars $\lambda_i$ in equation
\eqref{eq:lambda-i}.
Using the construction of the functor $\Rep^{\fg}$ in section
\ref{sec:rep-construction} shows that the Calabi-Yau category
$(\Rep^{\fg} \circ F)(A, \lambda)$ is given by the linear category
$\Rep^{\fg}(A)$, together with the Calabi-Yau structure given by the
composite of the Frobenius form with the Hattori-Stallings trace. For
a simple module $X_i$, this Calabi-Yau structure is given by
multiplying with the scalar $\lambda_i$ under the identification
$\End_{\Rep^{\fg}(A)}(X_i) \cong \K$. Thus, we have succeeded in
finding an equivalence $\eta$ as required in equation
\eqref{eq:eta-diagram}. Going through the equivalences $F$ and $G$, we
check that $\eta$ is even pseudo-natural. This shows that the diagram
\eqref{eq:rep-induced} commutes up to a pseudo-natural isomorphism. 
\end{proof}

\newpage
\newcommand{\etalchar}[1]{$^{#1}$}
\providecommand{\bysame}{\leavevmode\hbox to3em{\hrulefill}\thinspace}
\providecommand{\ZM}{\relax\ifhmode\unskip\space\fi Zbl }
\providecommand{\MR}{\relax\ifhmode\unskip\space\fi MR }
\providecommand{\arXiv}[1]{\relax\ifhmode\unskip\space\fi\href{http://arxiv.org/abs/#1}{arXiv:#1}}
\providecommand{\MRhref}[2]{%
  \href{http://www.ams.org/mathscinet-getitem?mr=#1}{#2}
}
\providecommand{\href}[2]{#2}

\end{document}